\documentclass[12pt,a4paper]{article}
\usepackage{amsmath,amsthm,amssymb, mathrsfs}
\usepackage{hyperref}
\usepackage{tikz}
\input xy
\xyoption{all}
\usepackage{epsfig}
\newtheorem{thm}{Theorem}[section]

\newtheorem{lem}[thm]{Lemma}

 \theoremstyle{definition}

\theoremstyle{remark}

\numberwithin{equation}{section}
\newcommand{\ben}{\begin{enumerate}}

\newcommand{\een}{\end{enumerate}}
\newcommand{\bit}{\begin{itemize}}
\newcommand{\eit}{\end{itemize}}

\begin{document}

\title
{Recovering Functions Defined on $\Bbb S^{n - 1}$ by Integration on Subspheres Obtained from Hyperplanes Tangent to a Spheroid}
\author{Yehonatan Salman \\ Weizmann Institute of Science \\ Email: salman.yehonatan@gmail.com}
\date{}
\maketitle

\begin{abstract}
The aim of this article is to introduce a method for recovering functions, defined on the $n - 1$ dimensional unit sphere $\Bbb S^{n - 1}$, using their spherical transform, which integrates functions on $n - 2$ dimensional subspheres, on a prescribed family of subspheres of integration. This family of subspheres is obtained as follows, we take a spheroid $\Sigma$ inside $\Bbb S^{n - 1}$ which contains the points $\pm e_{n}$ and then each subsphere of integration is obtained by the intersection of a hyperplane, which is tangent to $\Sigma$, with $\Bbb S^{n - 1}$. In particular, we obtain as a limiting case, by shrinking the spheroid into its main axis, a method for recovering functions in case where the subspheres of integration pass through a common point in $\Bbb S^{n - 1}$.
\end{abstract}

\section{Introduction and Motivation}

Recovering a function $f$, defined on a manifold $\Omega$, by integrating $f$ on a family $\Gamma$ of submanifolds of $\Omega$, in case when one can obtain a well-posed problem (i.e., when the dimension of the family $\Gamma$ is equal to the dimension of $\Omega$), is one of the main subjects of research in Integral Geometry. In many cases, a solution can be found by assuming some symmetric properties on the manifold $\Omega$ such as translation and rotation invariance.

In case where $\Omega$ is a sphere then one can use its special geometry in order to reconstruct a function $f$ in case where the family $\Gamma$ consists of subspheres of $\Omega$, where by a subsphere we mean a nonempty intersection of $\Omega$ with a hyperplane. If we assume, without loss of generality, that $\Omega$ is the unit sphere $\Bbb S^{n - 1}$, then the recovery problem for $\Omega$ was studied and solved in cases where the family $\Gamma$ of subspheres of integration has a specific geometric flavor. Some notable examples are when $\Gamma$ consists of great subspheres (i.e., intersections of hyperplanes which pass through the origin with $\Bbb S^{n - 1}$) (\cite{3, 5, 6, 7, 9, 11, 12, 13, 15}), of subspheres which pass through a common point which lies on $\Bbb S^{n - 1}$ (\cite{9, 13, 15}), of subspheres which are orthogonal to a subsphere of $\Bbb S^{n - 1}$ (\cite{2, 8, 10, 16}) and when $\Gamma$ consists of subspheres obtained by intersections of $\Bbb S^{n - 1}$ with hyperplanes which pass through a common point inside $\Bbb S^{n - 1}$(\cite{13, 14, 15}).

The main aim of this paper is to continue the research obtained in the above mentioned papers and obtain inversion procedures for families of subspheres of $\Bbb S^{n - 1}$ which have a specific geometry. In our case, each subsphere in the family $\Gamma$ is obtained by the intersection of $\Bbb S^{n - 1}$ with a hyperplane which is tangent to a fixed spheroid $\Sigma$ inside $\Bbb S^{n - 1}$ containing the north and south poles $\pm e_{n}$. In particular, we will show how by shrinking $\Sigma$ into its main axis one can obtain an inversion procedure for the case of the so called spherical slice transform (see \cite[Chapter 3, page 108]{9}) where the subspheres of integration pass through a common point $p$ which lies on $\Bbb S^{n - 1}$ (where in our case $p$ will be the south pole $-e_{n}$). It should be mentioned however that in this paper the solution for the above reconstruction problem is given as a series of functions rather than in a closed form. This is because the method used here includes, at some stage, an expansion into spherical harmonics. Expansion into spherical harmonics in our case can be used since the spheroid $\Sigma$ has a rotational symmetry with respect to its main axis. Of course, if $\Sigma$ is a general ellipsoid inside $\Bbb S^{n - 1}$ then one cannot use the method present here, the solution for this general problem is left for future research.

Our paper is organized as follows, in Chapter 2 we give all the necessary mathematical background for the formulation of the main result Theorem 2.1 and formulate the main result. In Chapter 3 we discuss the method behind the proof of Theorem 2.1 and show how the limiting case, where $\Sigma$ shrinks into its main axis, yields an inversion procedure for the spherical slice transform. In Chapter 4 we give the proof of Theorem 2.1. Chapter 5 is more technical and contains a characterization of the stereographic projections of the subspheres of integration and also contains a proof of the factorization of the infinitesimal volume measure, of each subsphere of integration, under the stereographic projection.

\section{Mathematical Background and the Main Result}

Denote by $\Bbb R^{n}$ the $n$ dimensional Euclidean space and by $\langle\hskip0.1cm,\hskip0.1cm\rangle$ the standard scalar product on $\Bbb R^{n}$. Denote by $\Bbb R^{+}$ the ray $[0,\infty)$, by $\Bbb S^{n - 1}$ the $n - 1$ dimensional unit sphere of $\Bbb R^{n}$, i.e., $\Bbb S^{n - 1} = \left\{x\in\Bbb R^{n}:|x| = 1\right\}$ and by $\omega_{n - 1} = 2\pi^{n / 2}/\Gamma(n / 2)$ the volume of $\Bbb S^{n - 1}$. Denote by $C\left(\Bbb S^{n - 1}\right)$ the set of continuous functions defined on $\Bbb S^{n - 1}$ and on $C\left(\Bbb S^{n - 1}\right)$ define the following inner product
$$\hskip-4.5cm\left\langle f_{1}, f_{2}\right\rangle_{\Bbb S^{n - 1}} = \int_{\Bbb S^{n - 1}}f_{1}(\psi)\overline{f_{2}(\psi)}d\psi, f_{1},f_{2}\in C\left(\Bbb S^{n - 1}\right)$$
where $d\psi$ is the standard infinitesimal volume measure on $\Bbb S^{n - 1}$.

For a point $\psi$ in $\Bbb S^{n - 1}$ define the following $n - 2$ dimensional subsphere of $\Bbb S^{n - 1}$:
$$\hskip-8cm \Bbb S_{\psi}^{n - 2} = \left\{x\in\Bbb S^{n - 1}:\langle x,\psi\rangle = 0\right\}.$$
For a fixed real number $\lambda > 0$ define the following spheroid in $\Bbb R^{n}$:
\begin{equation}\hskip-4.5cm\Sigma_{\lambda} = \{x\in\Bbb R^{n}:x_{n}^{2} + (x_{1}^{2} + ... + x_{n - 1}^{2})\cosh^{2}\lambda = 1\}.\end{equation}
Define the following stereographic and inverse stereographic projections respectively,
$$\hskip-4.3cm\Lambda:\Bbb S^{n - 1}\setminus\{e_{n}\} \rightarrow \Bbb R^{n - 1}, \Lambda(x) = \left(\frac{x_{1}}{1 - x_{n}},...,\frac{x_{n - 1}}{1 - x_{n}}\right),$$
$$\hskip-2cm\Lambda^{-1}:\Bbb R^{n - 1} \rightarrow \Bbb S^{n - 1}\setminus\{e_{n}\}, \Lambda(y) = \left(\frac{2y_{1}}{1 + |y|^{2}},...,\frac{2y_{n - 1}}{1 + |y|^{2}}, \frac{- 1 + |y|^{2}}{1 + |y|^{2}}\right).$$
We define the "stereographic projection" $f^{\ast}$ of a function $f$ in $C(\Bbb S^{n - 1})$ by
$$\hskip-8.75cm f^{\ast}:\Bbb R^{n - 1}\rightarrow\Bbb R, f^{\ast} = f\circ\Lambda^{-1}.$$
We will also define the function
\begin{equation}\hskip-5.4cm f^{\ast\ast}(x) = \frac{(f\circ\Lambda^{-1})(x)}{|x|^{n - 2}(1 + |x|^{2})^{n - 2}}\hskip0.1cm\textrm{where}\hskip0.1cm x\in\Bbb R^{n - 1}\setminus\{0\}.\end{equation}

Denote by $\mathcal{G}$ the isotropic group of rotations in $\Bbb S^{n - 1}$ which leave the unit vector $e_{n}$ fixed. That is
\vskip-0.2cm
$$\hskip-8.5cm \mathcal{G} = \{g\in SO(n):ge_{n} = e_{n}\}.$$

Define the Gegenbauer polynomials $C_{l}^{\lambda} = C_{l}^{\lambda}(t)$ of order $\lambda > -\frac{1}{2}$ and degree $l$ by the following orthogonality relations
\[
\hskip-4.1cm\int_{-1}^{1}C_{l}^{\lambda}(t)C_{k}^{\lambda}(t)(1 - t^{2})^{\lambda - \frac{1}{2}}dt =
\begin{cases}
0, \hskip2.1cm l\neq k,\\
\frac{2^{2\lambda - 1}\Gamma^{2}\left(\lambda + \frac{1}{2}\right)l!}{(l + \lambda)\Gamma(l + 2\lambda)}, l = k,
\end{cases}
\]
and for $\lambda = -\frac{1}{2}$ define $C_{l}^{-\frac{1}{2}}(t) = \cos(l\arccos(t))$.

For every integer $m\geq0$ define the following function $h_{m,\lambda}:\Bbb R^{+}\rightarrow\Bbb R$ by
\begin{equation}\hskip-13.5cm h_{m,\lambda}(x)\end{equation}
\[ \hskip-1.5cm = \begin{cases}
\left(2^{4 - n}\tanh^{3 - n}\lambda\right)xC_{m}^{\frac{n - 3}{3}}\left(\frac{x^{2} + 1 - \tanh^{2}\lambda}{2x}\right)\\
\left((1 + \tanh\lambda - x)(1 + \tanh\lambda + x)\right.\\
 \left.(x - 1 + \tanh\lambda)(x + 1 - \tanh\lambda)\right)^{\frac{n - 4}{2}}, 1 - \tanh\lambda \leq x \leq 1 + \tanh\lambda,\\
0, \hskip6.5cm o.w
\end{cases}
\]
For a function $F$, defined on $\Bbb R^{+}$, define the Mellin transform $\mathcal{M}F$ of $F$ by
$$\hskip-7.5cm(\mathcal{M}F)(s) = \int_{0}^{\infty}y^{s - 1}F(y)dy, s\in\Bbb C$$
where it should be noted that the above integral might not converge for every $s\in\Bbb C$.\\

For the Mellin transform we have the following inversion and convolution formulas for two functions $F_{1}$ and $F_{2}$ defined on $\Bbb R^{+}$ (see \cite{4}, Chapter 8.2 and 8.3):
\begin{equation}\hskip-5.7cm\mathcal{M}^{-1}(F_{1})(r) = \frac{1}{2\pi i}\int_{\varrho - i\infty}^{\varrho + i\infty}r^{-s}\mathcal{M}(F_{1})(s)ds,\end{equation}
\begin{equation}\hskip-5.75cm\mathcal{M}(F_{1}\star F_{2})(r) = \mathcal{M}(F_{1})(r)\mathcal{M}(F_{2})(1 - r)\end{equation}
where $\varrho\in(0,1)$ and where the convolution $F_{1}\star F_{2}$ is defined by
\vskip-0.2cm
$$\hskip-7cm (F_{1}\star F_{2})(s) = \int_{0}^{\infty}F_{1}(ss')F_{2}(s')ds'.$$
As we will show later, if $F_{1},F_{2}\in L^{1}(\Bbb R^{+})$ then formulas (2.4) and (2.5) are valid when $0 < \Re(r) < 1$.

For a function $f$ in $C\left(\Bbb S^{n - 1}\right)$ define its spherical transform $Sf$ to be the integral transform which integrates $f$ on $n - 2$ dimensional subspheres in $\Bbb S^{n - 1}$. That is,
$$\hskip-9.2cm (Sf)(\mathbf{\mathcal{S}}) = \int_{\mathbf{\mathcal{S}}}f(x)dS_{x}$$
where $dS_{x}$ is the standard infinitesimal volume measure on the subsphere $\mathbf{\mathcal{S}}$ of integration. Our aim is to recover functions in $C(\Bbb S^{n - 1})$ using their spherical transform where each subsphere of integration is obtained by the intersection of $\Bbb S^{n - 1}$ with a hyperplane which is tangent to the spheroid $\Sigma_{\lambda}$. Let us denote this family of subspheres by $\Upsilon_{\lambda}$. From Lemma 5.1 it follows that the family of $n - 2$ dimensional spheres $\Upsilon_{\lambda}^{\ast}$, which is obtained by projecting each subsphere in $\Upsilon_{\lambda}$ using the stereographic projection $\Lambda$, can be parameterized as in (5.1). Thus, by taking the inverse stereographic projection $\Lambda^{-1}$ we obtain the following parametrization for $\Upsilon_{\lambda}$:
\begin{equation}\hskip-3cm\Upsilon_{\lambda} = \underset{\psi\in\Bbb S^{n - 2}, c > 0}{\bigcup}\left\{\{\Lambda^{-1}\left(c\psi + c(\tanh\lambda)\omega\right):\omega\in\Bbb S^{n - 2}\}\right\}.\end{equation}
Thus, if we define the following $n - 2$ dimensional sphere in $\Bbb S^{n - 1}$
$$\hskip-5.3cm\mathrm{S}_{\psi,c} = \left\{\Lambda^{-1}\left(c\psi + c(\tanh\lambda)\omega\right), \omega\in\Bbb S^{n - 2}\right\}$$
then our data consists of the family of integrals
\begin{equation}\hskip-5.3cm(Sf)(\mathrm{S}_{\psi,c}) = \int_{\mathrm{S}_{\psi,c}}f(x)d\mathrm{S}_{\psi,c}, \psi\in\Bbb S^{n - 2}, c > 0\end{equation}
where $d\mathrm{S}_{\psi,c}$ is the standard measure on $\mathrm{S}_{\psi,c}$. We have the following result for the recovering of a function $f$ in $C\left(\Bbb S^{n - 1}\right)$ from the family of integrals (2.7).

\begin{thm}

Let $f$ be a function in $C\left(\Bbb S^{n - 1}\right)$ such that $f^{\ast\ast}$ as defined in (2.2) is in $L^{1}\left(\Bbb R^{n - 1}\setminus\{0\}\right)$ and such that
$$\hskip-6.7cm f^{\ast}(r\zeta) = \sum_{m = 0}^{\infty}\sum_{l = 1}^{d_{m}}f_{m,l}(r)Y_{l}^{m}(\xi)$$
is the spherical harmonic expansion of $f^{\ast} = f\circ\Lambda^{-1}$. Then, for each $m\geq0$ and $1\leq l\leq d_{m}$ the term $f_{m,l}$ can be recovered from the integral transform (2.7) as follows
$$\hskip-2.5cm f_{m,l}(r) = \frac{(1 + r^{2})^{n - 2}}{2\pi i}\int_{\varrho - i\infty}^{\varrho + i\infty}r^{-s}\frac{\mathcal{M}\left(\mathcal{K}_{m,l}\right)(s)}{\mathcal{M}\left(h_{m,\lambda}\right)(1 - s)}ds$$
where
$$\hskip-2.5cm \mathcal{K}_{m,l}(c) = \frac{1}{(2c\tanh\lambda)^{n - 2}\omega_{n - 3}}\int_{\Bbb S^{n - 2}}(Sf)(\mathrm{S}_{\psi,c})\overline{Y_{l}^{m}(\psi)}d\psi,$$
$h_{m,\lambda}$ is defined as in (2.3) and $\varrho$ is any number in the interval $(0, 1)$.

\end{thm}

\section{The Method Behind the Proof of Theorem 2.1 and the Limiting Case $\lambda\rightarrow\infty$}

The idea behind the proof of Theorem 2.1 consists mainly of three steps.

In the first step we use the stereographic projection in order to project our family $\Upsilon_{\lambda}$ of subspheres in $\Bbb S^{n - 1}$ into a family of hyperspheres in $\Bbb R^{n - 1}$. As Lemma 5.1 shows, the family $\Upsilon_{\lambda}$ is projection into a well defined family $\Upsilon_{\lambda}^{\ast}$ of hyperspheres in $\Bbb R^{n - 1}$. More specifically, each sphere in $\Upsilon_{\lambda}^{\ast}$ has its radius proportional, with the factor $\tan\lambda$, to the distance of its center from the origin. As Lemma 5.2 shows, the infinitesimal volume measure of each subsphere in $\Upsilon_{\lambda}$ is factored under the stereographic projection and thus we can make all of our analysis on $\Bbb R^{n - 1}$ with the family $\Upsilon_{\lambda}^{\ast}$ of spheres of integration.

In the second step we exploit the rotational invariance (with respect to the origin) of the family $\Upsilon_{\lambda}^{\ast}$ in order to reduce our problem to each term in the spherical harmonic expansion of the modified projection $g(x) = f^{\ast}(x) / (1 + |x|^{2})^{n - 2}$ of the function $f$ in question to be recovered. For each such term $g_{m,l}$ we obtain a convolution type equation relating $g_{m,l}$ to its corresponding term $\mathcal{K}_{m,l}$ in the expansion of the integral transform (2.7) into spherical harmonics. Using the inversion and convolution formulas for the Mellin transform one is able to express the term $g_{m,l}$ in terms of $\mathcal{K}_{m,l}$. Since we can extract each such term $g_{m,l}$ we can obviously recover $g$ and thus we can also recover $f^{\ast}$.

In the third and final step we just use the inverse stereographic projection on $f^{\ast}$ in order return to our original function $f$.\\

Observe that by Lemma 5.1 it follows that for $\lambda\rightarrow\infty$ the projected family of spheres $\Upsilon_{\lambda}^{\ast}$ has the following parametrization
$$\Upsilon_{\lambda}^{\ast} = \underset{\psi\in\Bbb S^{n - 2}, c > 0}{\bigcup}\left\{\{c\psi + c\omega:\omega\in\Bbb S^{n - 2}\}\right\}.$$
That is, $\Upsilon_{\lambda}^{\ast}$ consists of the hyperspheres in $\Bbb R^{n - 1}$ passing through the origin. By taking the inverse stereographic projection it is an easy exercise to show that the corresponding family $\Upsilon_{\lambda}$ of subspheres consists of all the subspheres which pass through the south pole $-e_{n}$. Hence, this limiting case yields an inversion procedure for the case where the subspheres of integration pass through a common point which lies on $\Bbb S^{n - 1}$. Observe that for this case the function $h_{m,\lambda}$ has the simpler form
\[ h_{m,\infty}(x) =
\begin{cases}
2^{4 - n}x^{n - 3}(4 - x^{2})^{\frac{n - 4}{2}}C_{m}^{\frac{n - 3}{3}}\left(\frac{x}{2}\right), 0 \leq x \leq 2,\\
0, \hskip5.15cm x\geq2.
\end{cases}
\]

\section{Proof of Theorem 2.1}
Denote $x = \Lambda^{-1}\left(c\psi + c(\tanh\lambda)\omega\right)$, then, by Lemma 5.2, $dx = d\mathrm{S}_{\psi,c}$ is given by formula (5.3). Hence we can write
\vskip-0.2cm
$$\hskip-12.5cm(Sf)(\mathrm{S}_{\psi,c})$$
$$ \hskip-1.5cm = \int_{\Bbb S^{n - 2}}\frac{(2c\tanh\lambda)^{n - 2}f\left(\Lambda^{-1}\left(c\psi + c(\tanh\lambda)\omega\right)\right)d\omega}{\left(1 + c^{2}\left(1 + \tanh^{2}\lambda + 2\langle\omega,\psi\rangle\tanh\lambda\right)\right)^{n - 2}}, \psi\in\Bbb S^{n - 2}, c > 0.$$
Hence, if we define
\vskip-0.2cm
\begin{equation}\hskip-5cm (Sf)^{\ast}(\mathrm{S}_{\psi,c}) = \frac{(Sf)(\mathrm{S}_{\psi,c})}{(2c\tanh\lambda)^{n - 2}}, g(x) = \frac{\left(f\circ\Lambda^{-1}\right)(x)}{(1 + |x|^{2})^{n - 2}},\end{equation}
then we have
\vskip-0.2cm
\begin{equation}\hskip-3.7cm (Sf)^{\ast}(\mathrm{S}_{\psi,c}) = \int_{\Bbb S^{n - 2}}g\left(c\psi + c(\tanh\lambda)\omega\right)d\omega, \psi\in\Bbb S^{n - 2}, c > 0.\end{equation}
Let us expand $g$ into spherical harmonics in $\Bbb R^{n - 1}$:
$$\hskip-6cm g(r\zeta) = \sum_{m = 0}^{\infty}\sum_{l = 1}^{d_{m}}g_{m,l}(r)Y_{l}^{m}(\zeta), r\geq0, \zeta\in\Bbb S^{n - 2}.$$
Observe that if $r\zeta = c\psi + c(\tanh\lambda)\omega$ then
$$r = c\sqrt{1 + \tanh^{2}\lambda + 2\tanh\lambda\langle\psi,\omega\rangle},\hskip0.1cm \zeta = \frac{\psi + (\tanh\lambda)\omega}{\sqrt{1 + \tanh^{2}\lambda + 2\tanh\lambda\langle\psi,\omega\rangle}}.$$
Hence, inserting the spherical harmonic expansion of $g$ into equation (4.2) we obtain
$$\hskip-3cm(Sf)^{\ast}(\mathrm{S}_{\psi,c}) = \sum_{m = 0}^{\infty}\sum_{l = 1}^{d_{m}}\int_{\Bbb S^{n - 2}}g_{m,l}\left(c\sqrt{1 + \tanh^{2}\lambda + 2\tanh\lambda\langle\psi,\omega\rangle}\right)$$ $$\hskip4cm \times Y_{l}^{m}\left(\frac{\psi + (\tanh\lambda)\omega}{\sqrt{1 + \tanh^{2}\lambda + 2\tanh\lambda\langle\psi,\omega\rangle}}\right)d\omega, \psi\in\Bbb S^{n - 2}, c > 0.$$
Making the change of variables
$$\hskip-3cm\omega = (\cos\varphi)\psi + (\sin\varphi)\eta, (\varphi,\eta)\in[0,\pi]\times\Bbb S_{\psi}^{n - 3}, d\omega = \sin^{n - 3}\varphi d\eta d\varphi$$
we have
$$\hskip-3cm(Sf)^{\ast}(\mathrm{S}_{\psi,c}) = \sum_{m = 0}^{\infty}\sum_{l = 1}^{d_{m}}\int_{0}^{\pi}\int_{\Bbb S_{\psi}^{n - 3}}g_{m,l}\left(c\sqrt{1 + \tanh^{2}\lambda + 2\tanh\lambda\cos\varphi}\right)$$ \begin{equation}\hskip2cm \times Y_{l}^{m}\left(\frac{(1 + \tanh\lambda\cos\varphi)\psi + (\tanh\lambda\sin\varphi)\eta}{\sqrt{1 + \tanh^{2}\lambda + 2\tanh\lambda\cos\varphi}}\right)\sin^{n - 3}\varphi d\eta d\varphi, \psi\in\Bbb S^{n - 2}, c > 0.\end{equation}
Observe that for constant $\varphi$ and $\lambda$ we can denote
$$\cos\xi = \frac{1 + \tanh\lambda\cos\varphi}{\sqrt{1 + \tanh^{2}\lambda + 2\tanh\lambda\cos\varphi}},  \sin\xi = \frac{\tanh\lambda\sin\varphi}{\sqrt{1 + \tanh^{2}\lambda + 2\tanh\lambda\cos\varphi}}$$
for some $\xi$ in $[-\pi,\pi]$. Since each $Y_{l}^{m}, m\geq0, 1\leq l\leq d_{m}$ is an eigenfunction of the Laplace-Beltrami operator in $\Bbb S^{n - 2}$ it follows (see for example \cite{1}) that
$$\hskip-3.5cm\int_{\Bbb S_{\psi}^{n - 3}}Y_{l}^{m}((\cos\xi)\psi + (\sin\xi)\eta)d\eta = \omega_{n - 3}C_{m}^{\frac{n - 3}{2}}(\cos\xi)Y_{l}^{m}(\psi).$$
Hence, from equation (4.3) we have
$$\hskip-1.5cm(Sf)^{\ast}(\mathrm{S}_{\psi,c}) = \omega_{n - 3}\sum_{m = 0}^{\infty}\sum_{l = 1}^{d_{m}}Y_{l}^{m}(\psi)\int_{0}^{\pi}g_{m,l}\left(c\sqrt{1 + \tanh^{2}\lambda + 2\tanh\lambda\cos\varphi}\right)$$ $$\hskip2cm \times C_{m}^{\frac{n - 3}{3}}\left(\frac{1 + \tanh\lambda\cos\varphi}{\sqrt{1 + \tanh^{2}\lambda + 2\tanh\lambda\cos\varphi}}\right)\sin^{n - 3}\varphi d\varphi, \psi\in\Bbb S^{n - 2}, c > 0.$$
From the orthonormality relations we have for $\left\{Y_{l}^{m}\right\}_{m\geq0, l = 1,...,d_{m}}$ on $\Bbb S^{n - 2}$ it follows that
\vskip-0.2cm
$$\hskip-2.5cm \frac{1}{\omega_{n - 3}}\langle (Sf)^{\ast}(\mathrm{S}_{\cdot,c}), Y_{l}^{m}\rangle_{\Bbb S^{n - 2}} = \int_{0}^{\pi}g_{m,l}\left(c\sqrt{1 + \tanh^{2}\lambda + 2\tanh\lambda\cos\varphi}\right)$$
\begin{equation}\hskip2cm \times C_{m}^{\frac{n - 3}{3}}\left(\frac{1 + \tanh\lambda\cos\varphi}{\sqrt{1 + \tanh^{2}\lambda + 2\tanh\lambda\cos\varphi}}\right)\sin^{n - 3}\varphi d\varphi, c > 0.\end{equation}
Let us make the following change of variables
$$x = \sqrt{1 + \tanh^{2}\lambda + 2\tanh\lambda\cos\varphi}, dx = \frac{\tanh\lambda\sin\varphi d\varphi}{\sqrt{1 + \tanh^{2}\lambda + 2\tanh\lambda\cos\varphi}},$$
in equation (4.4). Then, we have
\vskip-0.2cm
$$\hskip-10cm \frac{1}{\omega_{n - 3}}\langle (Sf)^{\ast}(\mathrm{S}_{\cdot,c}), Y_{l}^{m}\rangle_{\Bbb S^{n - 2}}$$
$$ \hskip-3cm = \frac{1}{2^{n - 4}\tanh^{n - 3}\lambda}\int_{1 - \tanh\lambda}^{1 + \tanh\lambda}xg_{m,l}\left(cx\right)
C_{m}^{\frac{n - 3}{3}}\left(\frac{x^{2} + 1 - \tanh^{2}\lambda}{2x}\right)$$ $$\left((1 + \tanh\lambda - x)(1 + \tanh\lambda + x)(x - 1 + \tanh\lambda)(x + 1 - \tanh\lambda)\right)^{\frac{n - 4}{2}}dx$$
\begin{equation} \hskip-8.5cm = \int_{0}^{\infty}g_{m,l}(cx)h_{m,\lambda}(x)dx, c > 0.\end{equation}
Thus, if we denote
$$\hskip-6cm \mathcal{K}_{m,l}(c) = \frac{1}{\omega_{n - 3}}\langle (Sf)^{\ast}(\mathrm{S}_{\cdot,c}), Y_{l}^{m}\rangle_{\Bbb S^{n - 2}}$$
then by using the Mellin convolution formula on equation (4.5) we obtain
\begin{equation}\mathcal{M}\left(\mathcal{K}_{m,l}\right)(s) = \mathcal{M}\left(g_{m,l}\right)(s)\mathcal{M}\left(h_{m,\lambda}\right)(1 - s).\end{equation}
In Lemma 5.3 it is proved that both $\mathcal{M}\left(g_{m,l}\right)$ and $\mathcal{M}\left(h_{m,\lambda}\right)(1 - \cdot)$ exist on the strip $0 < \Re(s) < 1$ and thus equation (4.6) is valid in this domain. Thus, after dividing both sides of equation (4.6) by $\mathcal{M}\left(h_{m,\lambda}\right)(1 - s)$ and using the Melling inversion formula and the fact that $\mathcal{M}\left(\mathcal{K}_{m,l}\right), \mathcal{M}\left(g_{m,l}\right)$ and $\mathcal{M}\left(h_{m,\lambda}\right)(1 - \cdot)$ are all defined on the strip $0 < \Re(s) < 1$ we have
$$g_{m,l}(r) = \frac{1}{2\pi i}\int_{\varrho - i\infty}^{\varrho + i\infty}r^{-s}\frac{\mathcal{M}\left(\mathcal{K}_{m,l}\right)(s)}{\mathcal{M}\left(h_{m,\lambda}\right)(1 - s)}ds $$
where $\varrho$ is any number in the interval $(0, 1)$. Using the relation (4.1) between $g$ and $f^{\ast} = f\circ\Lambda^{-1}$ proves Theorem 2.1.\hskip9cm$\square$
\section{Appendix}

\begin{lem}

The family $\Upsilon_{\lambda}^{\ast}$ of all the $n - 2$ dimensional spheres in $\Bbb R^{n - 1}$ obtained by projections, under $\Lambda$, of the $n - 2$ dimensional spheres obtained by intersections of $\Bbb S^{n - 1}$ with hyperplanes which are tangent to the spheroid $\Sigma_{\lambda}$, has the following parametrization

\begin{equation}\Upsilon_{\lambda}^{\ast} = \underset{\psi\in\Bbb S^{n - 2}, c > 0}{\bigcup}\left\{\{c\psi + c(\tanh\lambda)\omega:\omega\in\Bbb S^{n - 2}\}\right\}.\end{equation}
Furthermore, the projection of the intersection of $\Bbb S^{n - 1}$ with the unique hyperplane $H$ which is tangent to $\Sigma_{\lambda}$ and has the unit normal $n = ((\cos\theta)\psi, \sin\theta)$ ($\psi\in\Bbb S^{n - 2}, \theta\in\left[-\frac{\pi}{2}, \frac{\pi}{2}\right]$) is the $n - 2$ dimensional sphere in $\Bbb R^{n - 1}$ which corresponds, in the parametrization (5.1), to the subsphere with the parameters $c$ and $\psi$ where the relation between $c$ and $\theta$ is given by
\begin{equation}\hskip-3cm c = \frac{\cosh\lambda\cos\theta}{\sqrt{1 + \sinh^{2}\lambda\sin^{2}\theta} - \cosh\lambda\sin\theta},\hskip0.1cm \theta = \arctan\left(\frac{c^{2} - \cosh^{2}\lambda}{2c\cdot\cosh^{2}\lambda}\right).\end{equation}

\end{lem}

\begin{proof}
Since $\Sigma_{\lambda}$ is invariant with respect to rotations in the group $\mathcal{G}$, it follows that we can make our analysis first on the two dimensional plane $X_{1}X_{n}$ and then use rotations in $\mathcal{G}$ to obtain a complete parametrization of $\Upsilon_{\lambda}^{\ast}$. Hence, on the plane $X_{1}X_{n}$ and for $\theta\in\left[-\frac{\pi}{2},\frac{\pi}{2}\right]$ we want to find at which distance $t\geq0$ the line $x_{1}\cos\theta + x_{n}\sin\theta = t$ intersects the ellipse $x_{n}^{2} + (\cosh^{2}\lambda)x_{1}^{2} = 1$ at exactly one point. Extracting the variable $x_{n}$ from the first equation an inserting it in the second we obtain the following quadratic equation
\vskip-0.2cm
$$\hskip-2cm (\sin^{2}\theta\cosh^{2}\lambda + \cos^{2}\theta)x_{1}^{2} -2x_{1}t\cos\theta + t^{2} - \sin^{2}\theta = 0$$
and we need to find the value of $t$ for which the discriminant of the last equation is equal to zero. The discriminant of the last equation is zero when
$t = \sqrt{1 + \sinh^{2}\lambda\sin^{2}\theta} / \cosh\lambda$. Hence, the line
$$l_{\theta}:(x_{1}\cos\theta + x_{n}\sin\theta)\cosh\lambda = \sqrt{1 + \sinh^{2}\lambda\sin^{2}\theta}$$
is the unique line which is tangent to the ellipse $x_{n}^{2} + (\cosh^{2}\lambda)x_{1}^{2} = 1$ and its normal makes an angle $\theta$ with the $X_{1}$ axis. Now we want to find for each $\theta\in\left[-\frac{\pi}{2},\frac{\pi}{2}\right]$ at which points the line $l_{\theta}$ intersects the circle $x_{1}^{2} + x_{n}^{2} = 1$. By a direct substitution we obtain that the intersection points $p_{1} = (x_{1},y_{1})$ and $p_{2} = (x_{2},y_{2})$ are given by
$$\left(\frac{\left(\sinh\lambda\sin\theta + \sqrt{1 + \sinh^{2}\lambda\sin^{2}\theta}\right)\cos\theta}{\cosh\lambda}, \frac{\sqrt{1 + \sinh^{2}\lambda\sin^{2}\theta}\sin\theta - \sinh\lambda\cos^{2}\theta}{\cosh\lambda}\right),$$
$$\left(\frac{\left(-\sinh\lambda\sin\theta + \sqrt{1 + \sinh^{2}\lambda\sin^{2}\theta}\right)\cos\theta}{\cosh\lambda}, \frac{\sqrt{1 + \sinh^{2}\lambda\sin^{2}\theta}\sin\theta + \sinh\lambda\cos^{2}\theta}{\cosh\lambda}\right).$$

Thus, restricting the stereographic projection $\Lambda$ to the circle $\Bbb S^{1}$ in the plane $X_{1}X_{n}$ we have the following images for the points $p_{1}$ and $p_{2}$ under $\Lambda$:

$$\hskip-3.5cm p_{1}^{\ast} = \Lambda(p_{1}) = \frac{\left(\sinh\lambda\sin\theta + \sqrt{1 + \sinh^{2}\lambda\sin^{2}\theta}\right)\cos\theta}{\cosh\lambda - \sqrt{1 + \sinh^{2}\lambda\sin^{2}\theta}\sin\theta + \sinh\lambda\cos^{2}\theta},$$
$$\hskip-3.5cm p_{2}^{\ast} = \Lambda(p_{2}) = \frac{\left(-\sinh\lambda\sin\theta + \sqrt{1 + \sinh^{2}\lambda\sin^{2}\theta}\right)\cos\theta}{\cosh\lambda - \sqrt{1 + \sinh^{2}\lambda\sin^{2}\theta}\sin\theta - \sinh\lambda\cos^{2}\theta}.$$

Returning back to the whole space $\Bbb R^{n}$, it follows that the image, under $\Lambda$, of the intersection of the unit sphere $\Bbb S^{n - 1}$ with the unique hyperplane $H$ which is tangent to the spheroid $\Sigma_{\lambda}$ and has a unit normal $n = e_{1}\cos\theta + e_{n}\sin\theta$, is the $n - 2$ dimensional sphere in $\Bbb R^{n - 1}$ which has a center at
$$\hskip-1.5cm C_{H} = \left(\frac{1}{2}\left(p_{1}^{\ast} + p_{2}^{\ast}\right),\overline{0}\right) = \left(\frac{\cosh\lambda\cos\theta}{\sqrt{1 + \sinh^{2}\lambda\sin^{2}\theta} - \cosh\lambda\sin\theta},\overline{0}\right)$$
and radius
$$\hskip-1.5cm R_{H} = \frac{1}{2}\left|p_{1}^{\ast} - p_{2}^{\ast}\right| = \frac{\sinh\lambda\cos\theta}{\sqrt{1 + \sinh^{2}\lambda\sin^{2}\theta} - \cosh\lambda\sin\theta}.$$
Observe that for every $\theta\in\left[-\frac{\pi}{2},\frac{\pi}{2}\right]$ we have the relation $R_{H} = \tanh\lambda\cdot\left|C_{H}\right|$ and thus the above sphere has the following parametrization
$$\hskip-4.5cm \mathrm{S}_{\psi,c}^{\ast}:ce_{1} + c(\tanh\lambda)\omega, \omega\in\Bbb S^{n - 2}$$
such that the relation between $c$ and $\theta$ is given by
$$\hskip-3cm c = \frac{\cosh\lambda\cos\theta}{\sqrt{1 + \sinh^{2}\lambda\sin^{2}\theta} - \cosh\lambda\sin\theta},\hskip0.1cm \theta = \arctan\left(\frac{c^{2} - \cosh^{2}\lambda}{2c\cdot\cosh^{2}\lambda}\right).$$
Now, we want to find for a general unit normal $n = ((\cos\theta)\psi, \sin\theta)$, $\left(\psi\in\Bbb S^{n - 2}, \theta\in\left[-\frac{\pi}{2}, \frac{\pi}{2}\right]\right)$ the image, under $\Lambda$, of the intersection of the unit sphere $\Bbb S^{n - 1}$ with the unique hyperplane $H$ which is tangent to the spheroid $\Sigma_{\lambda}$ and has the unit normal $n$. From the rotational symmetry of the spheroid $\Sigma_{\lambda}$ it follows that this is the $n - 2$ dimensional sphere $\mathrm{S}_{\psi,c}^{\ast}$ in $\Bbb R^{n - 1}$ which has the following parametrization:
\vskip-0.35cm
$$\hskip-4.5cm \mathrm{S}_{\psi,c}:c\psi + c(\tanh\lambda)\omega, \omega\in\Bbb S^{n - 2}$$
where again the relation between $c$ and $\theta$ is given by (5.2).

\end{proof}

\begin{lem}

Let $\mathrm{S}_{\psi,c}$ be the $n - 2$ dimensional subsphere of $\Bbb S^{n - 1}$ given by the following parametrization

$$\hskip-5.5cm \mathrm{S}_{\psi,c}:\Lambda^{-1}\left(c\psi + c(\tanh\lambda)\omega\right),\omega\in\Bbb S^{n - 2},$$
then
\begin{equation}\hskip-3cm d\mathrm{S}_{\psi,c} = \frac{(2c\tanh\lambda)^{n - 2}d\omega}{\left(1 + c^{2}\left(1 + \tanh^{2}\lambda + 2\langle\omega,\psi\rangle\tanh\lambda\right)\right)^{n - 2}}.\end{equation}

\end{lem}

\begin{proof}

From the definition of the inverse stereographic projection $\Lambda^{-1}$ it follows that $\mathrm{S}_{\psi,c}$ has the following parametrization
$$\hskip-6cm\mathrm{S}_{\psi,c} = \left\{\left(\frac{2c(\psi + (\tanh\lambda)\omega)}{1 + c^{2}\left(1 + \tanh^{2}\lambda + 2\langle\psi,\omega\rangle\tanh\lambda\right)},\right.\right.$$ $$\hskip-1.5cm\left.\left.  \frac{- 1 + c^{2}\left(1 + \tanh^{2}\lambda + 2\langle\psi,\omega\rangle\tanh\lambda\right)}{1 + c^{2}\left(1 + \tanh^{2}\lambda + 2\langle\psi,\omega\rangle\tanh\lambda\right)}\right):\omega\in\Bbb S^{n - 2}\right\}.$$
Making the following parametrization for $\omega$:
$$\hskip-4cm\omega = (\cos\rho)\psi + (\sin\rho)\psi^{\ast}, \psi^{\ast}\in\Bbb S_{\psi}^{n - 3}, \rho\in[0,\pi],$$
we have the following parametrization for $\mathrm{S}_{\psi,c}$
$$\hskip-6cm\mathrm{S}_{\psi,c} = \left\{\left(\frac{2c((1 + \cos\rho\tanh\lambda)\psi + (\sin\rho\tanh\lambda)\psi^{\ast})}{1 + c^{2}\left(1 + \tanh^{2}\lambda + 2(\cos\rho)\tanh\lambda\right)},\right.\right.$$ \begin{equation}\hskip1cm\left.\left.  \frac{- 1 + c^{2}\left(1 + \tanh^{2}\lambda + 2(\cos\rho)\tanh\lambda\right)}{1 + c^{2}\left(1 + \tanh^{2}\lambda + 2(\cos\rho)\tanh\lambda\right)}\right):\psi^{\ast}\in\Bbb S_{\psi}^{n - 3}, \rho\in[0,\pi]\right\}.\end{equation}

By Lemma 5.1 it follows that the subsphere $\mathrm{S}_{\psi,c}$ is the intersection of $\Bbb S^{n - 1}$ with the unique hyperplane which is tangent to $\Sigma_{\lambda}$ and has the unit normal $n = ((\cos\theta)\psi, \sin\theta)$ where the relation between $\theta$ and $c$ is given by equation (5.2). Since the variable $\psi$ in the unit normal $n$ comes from rotations in the group $\mathcal{G}$ and since the infinitesimal measure of any subset in $\Bbb S^{n - 1}$ is invariant with respect to rotations in $\mathcal{G}$, it follows that $\mathrm{S}_{\psi,c}$ is independent of $\psi$. Hence, we will assume from now on that $\psi = e_{n - 1}$. In this case observe that for the following rotation matrix
$$\hskip-6cm A_{\theta} = \left(\begin{array}{cccccc}
1 & 0 & ... & 0 & 0 & 0\\
0 & 1 & ... & 0 & 0 & 0\\
& .............\\
0 & 0 & ... & 1 & 0 & 0\\
0 & 0 & ... & 0 & -\sin\theta & \cos\theta\\
0 & 0 & ... & 0 & \cos\theta & \sin\theta
\end{array}\right) $$
we have
$$\hskip-6cm A_{\theta}\left(\mathrm{S}_{e_{n - 1},c}\right) = \left\{\left(\frac{2c(\sin\rho\tanh\lambda)\psi^{\ast}}{1 + c^{2}\left(1 + \tanh^{2}\lambda + 2(\cos\rho)\tanh\lambda\right)},\right.\right.$$ $$\hskip1cm \frac{\cos\theta\sinh\lambda}{\cosh^{3}\lambda}\cdot\frac{\cos\rho\cosh^{2}\lambda + (2\sinh\lambda\cosh\lambda + (2\cosh^{2}\lambda - 1)\cos\rho)c^{2}}{1 + c^{2}\left(1 + \tanh^{2}\lambda + 2(\cos\rho)\tanh\lambda\right)},$$
\begin{equation} \left.\left. \hskip-3.5cm\frac{\left(\cosh^{2}\lambda + c^{2}\right)\cos\theta}{2c\cosh^{2}\lambda}\right):\psi^{\ast}\in\Bbb S_{e_{n - 1}}^{n - 3}, \rho\in[0,\pi]\right\}\end{equation}
where we used the relation (5.2) between the variables $\theta$ and $c$. Denote
\begin{equation}\hskip-2cm \mathrm{G}(\lambda, c) = \sqrt{1 - \frac{\left(\cosh^{2}\lambda + c^{2}\right)^{2}\cos^{2}\theta}{4c^{2}\cosh^{4}\lambda}} = \frac{2c\cosh\lambda\sinh\lambda}{\sqrt{4c^{2}\cosh^{4}\lambda + (c^{2} - \cosh^{2}\lambda)^{2}}},\end{equation}
then dividing the parametrization (5.5) of $A_{\theta}\left(\mathrm{S}_{e_{n - 1},c}\right)$ by $\mathrm{G}(\lambda, c)$ we have
$$\hskip-10.5cm \frac{1}{\mathrm{G}(\lambda, c)}A_{\theta}\left(\mathrm{S}_{e_{n - 1},c}\right)$$
$$ \hskip-5.5cm = \left\{\left(\frac{1}{\mathrm{G}(\lambda, c)}\cdot\frac{2c(\sin\rho\tanh\lambda)\psi^{\ast}}{1 + c^{2}\left(1 + \tanh^{2}\lambda + 2(\cos\rho)\tanh\lambda\right)},\right.\right.$$ $$\hskip1cm \frac{1}{\mathrm{G}(\lambda, c)}\cdot\frac{\cos\theta\sinh\lambda}{\cosh^{3}\lambda}\cdot\frac{\cos\rho\cosh^{2}\lambda + (2\sinh\lambda\cosh\lambda + (2\cosh^{2}\lambda - 1)\cos\rho)c^{2}}{1 + c^{2}\left(1 + \tanh^{2}\lambda + 2(\cos\rho)\tanh\lambda\right)},$$
\begin{equation} \left.\left. \hskip-5.5cm \frac{\sqrt{1 - \mathrm{G}^{2}(\lambda, c)}}{\mathrm{G}(\lambda, c)}\right):\psi^{\ast}\in\Bbb S_{e_{n - 1}}^{n - 3}, \rho\in[0,\pi]\right\}.\end{equation}
Observe that the right hand side of the parametrization (5.7) is of the form
\begin{equation}\hskip-12cm\left(\sqrt{1 - r^{2}}\psi^{\ast},r, C\right)\end{equation}
where
$$r(\rho) = \frac{1}{\mathrm{G}(\lambda, c)}\cdot\frac{\cos\theta\sinh\lambda}{\cosh^{3}\lambda}\cdot\frac{\cos\rho\cosh^{2}\lambda + (2\sinh\lambda\cosh\lambda + (2\cosh^{2}\lambda - 1)\cos\rho)c^{2}}{1 + c^{2}\left(1 + \tanh^{2}\lambda + 2(\cos\rho)\tanh\lambda\right)}$$
and $C$ is a constant which does not depend on $r$ and $\psi^{\ast}$. Since $C$ does not depend on $r$ or $\psi^{\ast}$ it can be easily verified that the infinitesimal volume measure given by the parametrization (5.8) is $(1 - r^{2})^{\frac{n - 4}{2}}d\psi^{\ast}dr$.
Since
\vskip-0.2cm
$$\hskip-1.5cm\frac{dr}{d\rho} = \frac{1}{\mathrm{G}(\lambda, c)}\cdot\frac{\cos\theta\sinh\lambda}{\cosh^{5}\lambda}\frac{\left(\cosh^{4}\lambda + 2c^{2}\cosh(2\lambda)\cosh^{2}\lambda + c^{4}\right)\sin\rho}{\left(1 + c^{2}\left(1 + \tanh^{2}\lambda + 2(\cos\rho)\tanh\lambda\right)\right)^{2}}$$
it follows that
\vskip-0.2cm
$$\hskip-10.5cm d\left(\frac{1}{\mathrm{G}(\lambda, c)}A_{\theta}\left(\mathrm{S}_{e_{n - 1},c}\right)\right)$$
\begin{equation} = \frac{d\rho d\psi^{\ast}\sin^{n - 3}\rho}{\mathrm{G}^{n - 3}(\lambda, c)}\frac{\cos\theta\sinh\lambda}{\cosh^{5}\lambda}\frac{(2c\tanh\lambda)^{n - 4}\left(\cosh^{4}\lambda + 2c^{2}\cosh(2\lambda)\cosh^{2}\lambda + c^{4}\right)}{\left(1 + c^{2}\left(1 + \tanh^{2}\lambda + 2(\cos\rho)\tanh\lambda\right)\right)^{n - 2}}.\end{equation}
Observe that since $\omega = (\cos\rho)e_{n - 1} + (\sin\rho)\psi^{\ast}$ we have $d\omega = \sin^{n - 3}\rho d\rho d\psi^{\ast}$ and since the rotation matrix $A_{\theta}$ does not change the infinitesimal measure in the left hand side of equation (5.9) it follows that
\vskip-0.2cm
$$\hskip-12.5cm\frac{d\mathrm{S}_{e_{n - 1},c}}{\mathrm{G}^{n - 2}(\lambda, c)}$$
\begin{equation} \hskip-0.5cm = \frac{d\omega}{\mathrm{G}^{n - 3}(\lambda, c)}\frac{\cos\theta\sinh\lambda}{\cosh^{5}\lambda}\frac{(2c\tanh\lambda)^{n - 4}\left(\cosh^{4}\lambda + 2c^{2}\cosh(2\lambda)\cosh^{2}\lambda + c^{4}\right)}{\left(1 + c^{2}\left(1 + \tanh^{2}\lambda + 2(\cos\rho)\tanh\lambda\right)\right)^{n - 2}}.\end{equation}
Multiplying equation (5.10) by $\mathrm{G}^{n - 2}(\lambda,c)$ and using the explicit formula (5.6) for $\mathrm{G}(\lambda,c)$ we obtain that
\vskip-0.2cm
$$\hskip-13cm d\mathrm{S}_{e_{n - 1},c}$$
$$ \hskip-1.6cm = \left(\frac{(2c)^{n - 3}\cos\theta\sinh^{n - 2}\lambda}{\cosh^{n}\lambda}\frac{\left(\cosh^{4}\lambda + 2c^{2}\cosh(2\lambda)\cosh^{2}\lambda + c^{4}\right)}{\left(1 + c^{2}\left(1 + \tanh^{2}\lambda + 2(\cos\rho)\tanh\lambda\right)\right)^{n - 2}}\right)$$ $$\hskip-8.2cm \times\frac{d\omega}{\sqrt{4c^{2}\cosh^{4}\lambda + (c^{2} - \cosh^{2}\lambda)^{2}}}.$$
Using the fact that
\vskip-0.2cm
$$\hskip-6cm \cos\theta = \frac{2c\cosh^{2}\lambda}{\sqrt{4c^{2}\cosh^{4}\lambda + (c^{2} - \cosh^{2}\lambda)^{2}}}$$
and that $\cos\rho = \langle\omega,\psi\rangle$ we obtain Lemma 5.2.

\end{proof}

\begin{lem}
Let $g$ be a function defined on $\Bbb R^{n - 1}$ such that $g(x) / |x|^{n - 2}$ belongs to $L_{1}(\Bbb R^{n - 1}\setminus\{0\})$. Let
\vskip-0.2cm
$$\hskip-6cm g(r\zeta) = \sum_{m = 0}^{\infty}\sum_{l = 1}^{d_{m}}g_{m,l}(r)Y_{l}^{m}(\zeta), r\geq0, \zeta\in\Bbb S^{n - 2}$$
be the expansion of $g$ into spherical harmonics in $\Bbb R^{n - 1}$. Then, for every $m\geq0, 1\leq l \leq d_{m}$ the Mellin transform $\mathcal{M}(g_{m,l})$ of $g_{m,l}$ exists in the strip $0 < \Re(\rho) < 1$. Also, if $h_{m,\lambda}$ is defined as in (2.3), then $\mathcal{M}(h_{m,\lambda})(1 - \cdot)$ exists in the same strip.
\end{lem}

\begin{proof}
By the definition of the Mellin transform, we have
$$\hskip-4cm\mathcal{M}(g_{m,l})(s) = \int_{0}^{\infty}y^{s - 1}g_{m,l}(y)dy.$$
Assume that $0 < \Re(s) < 1$, then $s = a + ib$ where $0 < a < 1$, $b\in\Bbb R$ and thus
$$\left|\mathcal{M}(g_{m,l})(s)\right| = \left|\int_{0}^{\infty}y^{s - 1}g_{m,l}(y)dy\right| \leq \int_{0}^{\infty}\left|y^{s - 1}\right||g_{m,l}(y)|dy$$
$$ = \int_{0}^{\infty}\left|y^{a + ib - 1}\right||g_{m,l}(y)|dy.$$
Since $y^{ib} = \exp(ib\log(y))$ and $\log(y)$ is real when $y > 0$, it follows that $|y^{ib}| = 1$. Thus
\begin{equation}
\left|\mathcal{M}(g_{m,l})(s)\right| \leq \int_{0}^{\infty}\left|y^{a - 1}\right||g_{m,l}(y)|dy \leq \int_{0}^{1}\left|y^{a - 1}\right||g_{m,l}(y)|dy + \int_{1}^{\infty}\left|y^{a - 1}\right||g_{m,l}(y)|dy.
\end{equation}
In Lemma 5.4 (see the end of this lemma) it is proved that if $g(x) / |x|^{n - 2}$ belongs to $L_{1}(\Bbb R^{n - 1}\setminus\{0\})$, then $g_{m,l}\in L_{1}(\Bbb R^{+})$ for $m\geq0, 1\leq l \leq d_{m}$. Thus, the first integral in the right hand side of equation (5.11) converges since $- 1 < a - 1 < 0$ and $g_{m,l}$ is bounded in the interval $(0,1)$ and the second integral converges since $g_{m,l}\in L_{1}(\Bbb R^{+})$ and $\left|y^{a - 1}\right||g_{m,l}(y)|\leq|g_{m,l}(y)|$ when $y\rightarrow\infty$.

By the definition of $h_{m,\lambda}$, it follows easily that $h_{m,\lambda}\in L_{1}(\Bbb R^{+})$ when the dimension $n$ is greater than or equals to $2$. Thus it can be proved in the same way that the Mellin transform $\mathcal{M}(h_{m,\lambda})$ of $h_{m,\lambda}$ exists on the strip $0 < \Re(\rho) < 1$ which is equivalent to the existence of $\mathcal{M}(h_{m,\lambda})(1 - \cdot)$ on the strip $0 < \Re(\rho) < 1$.
\end{proof}

\begin{lem}
Let $g$ be a function in $C^{\infty}(\Bbb R^{n - 1})$ such that $g(x) / |x|^{n - 2}$ belongs to $L_{1}(\Bbb R^{n - 1}\setminus\{0\})$. Let
\vskip-0.2cm
$$\hskip-6cm g(r\zeta) = \sum_{m = 0}^{\infty}\sum_{l = 1}^{d_{m}}g_{m,l}(r)Y_{l}^{m}(\zeta), r\geq0, \zeta\in\Bbb S^{n - 2}$$
be the expansion of $g$ into spherical harmonics in $\Bbb R^{n - 1}$. Then for every $m\geq0, 1\leq l \leq d_{m}$, $g_{m,l}\in L_{1}(\Bbb R^{+})$.
\end{lem}

\begin{proof}
From the orthogonality condition for spherical harmonics, we have
$$\hskip-4cm g_{m,l}(r) = \int_{\Bbb S^{n - 2}}g(r\zeta)Y_{l}^{m}(\zeta)dS(\zeta).$$
Thus,
$$\hskip-2cm \int_{0}^{\infty}|g_{m,l}(r)|dr = \int_{0}^{\infty}\left|\int_{\Bbb S^{n - 2}}g(r\zeta)Y_{l}^{m}(\zeta)dS(\zeta)\right|dr$$
$$ \leq \int_{0}^{\infty}\int_{\Bbb S^{n - 2}}|g(r\zeta)Y_{l}^{m}(\zeta)|dS(\zeta)dr\leq M_{l,m}\int_{0}^{\infty}\int_{\Bbb S^{n - 2}}|g(r\zeta)|dS(\zeta)dr$$
where $M_{l,m}$ is an upper bound for $|Y_{l}^{m}(\zeta)|, \zeta\in \Bbb S^{n - 2}$. Making the change of variables $x = r\zeta$, $dx = r^{n - 2}dS(\zeta)dr$ yields
\begin{equation}
\int_{0}^{\infty}|g_{m,l}(r)|dr \leq M_{l,m}\int_{\Bbb R^{n - 1}}\frac{|g(x)|}{|x|^{n - 2}}dx
\end{equation}
and since $g(x) / |x|^{n - 2}$ belongs to $L_{1}(\Bbb R^{n - 1}\setminus\{0\})$, it follows that the integral in the right hand side of equation (5.12) converges which proves the lemma.
\end{proof}


\begin{thebibliography}{}

\bibitem{1} M. Agranovsky, V. V. Volchkov and L. A. Zalcman. \textit{Conical uniqueness sets for the spherical Radon transform}, Bulletin of the London Mathematical Society, 31 (1999), 231-236.

\bibitem{2} Y. A. Antipov, R. Estrada and B. Rubin. \textit{Inversion formulas for spherical means
in constant curvature spaces}, Journal D'Anal. Math. 118 (2012), 623-656.

\bibitem{3} T. N. Bailey, M. G. Eastwood, A. R. Gover, and L. J. Mason. \textit{Complex analysis and the Funk transform},
Journal of the Korean Mathematical Society, 40, no. 4. (2003) 577-593.

\bibitem{4} B. Dambaru and L. Debnath. \textit{Integral Transforms and Their Applications}, CRC Press, New York, 2007.

\bibitem{5} P. Funk. \textit{U¨ber eine geometrische Anwendung der Abelschen Integralgleichung}, Math. Ann. 77(1). (1915) 129-135.

\bibitem{6} P. Funk. \textit{U¨ber Fla¨chen mit lauter geschlossenen geoda¨tischen Linien} (on surfaces with louder closed geodesic lines), Math. Ann. 74 (1913), 278-300.

\bibitem{7} I. M. Gelfand, S. G. Gindikin, and M. I. Graev. \textit{Integral geometry in affine and
projective spaces}, Itogi Nauki Tekh., Ser. Sovrem. Probl. Mat. 16 (1980), 53-226,
translated into English in J. Sov. Math. 18, 39-167, 1980.

\bibitem{8} S. Gindikin, J. Reeds, and L. Shepp. Spherical tomography and spherical integral geometry.
In E. T. Quinto, M. Cheney, and P. Kuchment, editors. \textit{Tomography, Impedance Imaging,
and Integral Geometry}, volume 30 of Lectures in Appl. Math, pages 83-92. South Hadley,
Massachusetts, 1994.

\bibitem{9} S. Helgason. \textit{The Radon Transform}, Birkh¨auser, Basel, 1980.

\bibitem{10} R. Hielscher and M. Quellmalz. \textit{Reconstructing a function on the sphere from its means along
vertical slices}, Preprint 2015-10, Faculty of Mathematics, Technische Universit¨at Chemnitz,
2015.

\bibitem{11} H. Minkowski. \textit{\"{U}ber die K\"{o}rper Konstanter Breite}, Moskau Mathematische Sammlung 25 (1904-06), 505-508.

\bibitem{12} V. P. Palamodov. \textit{Reconstruction from a sampling of circle integrals in SO(3)},
Inverse Problem 26 (2010), no. 9, 095-008, 10 pp.

\bibitem{13} M. Quellmalz. \textit{A generalization of the Funk-Radon transform}, Inverse Problems 33, no. 3, 025013, 2017.

\bibitem{14} Y. Salman. \textit{An inversion formula for the spherical transform in $S^{2}$ for a special family of circles of integration}, Anal. Math.Phys., 6(1):43 – 58, 2016.

\bibitem{15} Y. Salman. \textit{Recovering functions defined on the unit sphere by integration on a special family
of sub-spheres}, Anal. Math. Phys., Advance online publication, 2016.

\bibitem{16} G. Zangerl and O. Scherzer. \textit{Exact reconstruction in photoacoustic tomography with circular
integrating detectors II: Spherical geometry}, Math. Methods Appl. Sci., 33(15):1771-1782.
2010.


\end{thebibliography}
\end{document}